\documentclass[11pt]{article}
\usepackage{amsmath, amsthm, amssymb,latexsym, enumerate}
\usepackage{graphicx,kotex}
\usepackage{hyperref}
\usepackage{xcolor}
\usepackage{multirow}
\usepackage{cite}\usepackage{booktabs}
\usepackage{graphicx, caption, subcaption}
\usepackage{subcaption}
\usepackage{caption}
\usepackage{colortbl}
\usepackage{float}
\usepackage{tikz}
\usepackage[normalem]{ulem}
\usetikzlibrary{positioning}
\usetikzlibrary{decorations.pathreplacing}
\tikzstyle{c} = [draw, every node/.style={circle, draw, inner sep=1.5pt,fill=white}]
\tikzstyle{d} = [draw, every node/.style={circle, draw, inner sep=2pt,fill=white,font=\small}]

\usetikzlibrary{patterns,calc}
\tikzset{super thick/.style= {line width=2.4pt}}

\numberwithin{equation}{section}
\usepackage{array}
\newcolumntype{L}[1]{>{\raggedright\let\newline\\\arraybackslash\hspace{0pt}}m{#1}}
\newcolumntype{C}[1]{>{\centering\let\newline\\\arraybackslash\hspace{0pt}}m{#1}}
\newcolumntype{R}[1]{>{\raggedleft\let\newline\\\arraybackslash\hspace{0pt}}m{#1}}

\definecolor{myblue}{rgb}{0,0.4,0.8}
\definecolor{mygreen}{rgb}{0, 0.8, 0.4}
\definecolor{myred}{rgb}{204, 0, 0}
\definecolor{brown}{rgb}{0.6, 0.4, 0}

\newtheorem{theorem}{Theorem}[section]
\newtheorem{lemma}[theorem]{Lemma}
\newtheorem{corollary}[theorem]{Corollary}

\newtheorem{claim}[theorem]{Claim}

\newtheorem{conjecture}[theorem]{Conjecture}

\newtheorem{observation}[theorem]{Observation}

\counterwithin*{case}{section}

\title{Minimum degree conditions for removable matchings in $k$-connected graphs}

\author{Hojin Chu
\thanks{School of Computational Sciences, Korea Institute for Advanced Study (KIAS), Seoul 02455, Republic of Korea. E-mail: {\tt hojinchu@kias.re.kr}.}
\and Ringi Kim
\thanks{Department of Mathematics, Inha University, Incheon 22212, Republic of Korea. E-mail: {\tt ringikim@inha.ac.kr}.}
\and Boram Park
\thanks{Department of Mathematics Education, Seoul National University, Seoul 08826, Republic of Korea. E-mail: {\tt borampark@snu.ac.kr}.}
}

\date{}

\begin{document}

\maketitle

\begin{abstract}
In 1969, Halin proved that every $k$-connected graph $G$ with minimum degree at least $k+1$ contains an edge $e$ such that $G-e$ is $k$-connected. As an edge is a matching of size one, it is natural to ask whether Halin's result extends to matchings of larger size, a question recently investigated by Li, Zhou, Fujita, and Mao. A matching $M$ of a $k$-connected graph $G$ is called \emph{$k$-removable} if $G-M$ is $k$-connected.  In this paper, we study minimum degree conditions that guarantee the existence of a $k$-removable matching of prescribed size. Specifically, we prove that for all positive integers $k$ and $m$, every $k$-connected graph $G$ with at least $2m$ vertices contains a $k$-removable matching of size $m$ if  \[\delta(G)\ \ge\ \begin{cases} \max\bigl\{k+\bigl\lceil\tfrac m2\bigr\rceil,\ 2m\bigr\} & \text{if } k\ge m,\\[2pt] k+m & \text{if } k<m. \end{cases}\] As a consequence, every $k$-connected graph $G$ with $\delta(G)\ge2k+1$ contains a $k$-removable matching of size $\bigl\lceil(\delta(G)+1)/2\bigr\rceil$, unless $\delta(G)$ is even and $G\cong K_{\delta(G)+1}$.  This verifies a conjecture of Li, Zhou, Fujita, and Mao in the range $\delta(G)\ge2k+1$. Our main tool, of independent interest, is a strengthening of Halin's result producing a $k$-removable edge that avoids a prescribed set of vertices.
\end{abstract}

\section{Introduction}

All graphs in this paper are finite, undirected, and simple. For a graph $G$, we denote its vertex and edge sets by $V(G)$ and $E(G)$, respectively. For a vertex $v\in V(G)$, let $N_G(v)$ and $d_G(v)$ denote the neighborhood and degree of $v$, respectively, and let $\delta(G)$ denote the minimum degree of $G$. For $X\subseteq V(G)$, we write $G-X$ for the graph obtained by deleting the vertices in $X$ and all incident edges, and $G[X]$ for the graph induced by $X$. For $F\subseteq E(G)$, we write $G-F$ for the graph obtained from $G$ by deleting the edges in $F$, and abbreviate $G-\{e\}$ as $G-e$.

An edge $e$ of a $k$-connected graph $G$ is \emph{$k$-removable} if $G-e$ is $k$-connected. In 1969, Halin \cite{halin} determined the minimum degree condition that guarantees the existence of a removable edge.

\begin{theorem}[Halin \cite{halin}]\label{thm:halin}
For every $k\ge 1$, every $k$-connected graph $G$ with $\delta(G)\ge k+1$
contains a $k$-removable edge.
\end{theorem}

The bound in Theorem~\ref{thm:halin} is sharp, since no edge of a $k$-regular $k$-connected graph is $k$-removable. 
This result initiated an extensive line of research on \emph{connectivity-keeping subgraphs}, in which one seeks a prescribed subgraph whose vertices or edges can be deleted while preserving $k$-connectedness.

The vertex-deletion version was initiated by Chartrand, Kaugars, and Lick \cite{ckl} and was subsequently developed, most notably by Mader \cite{maderpath,madertree}.
Since then, it has been extensively studied from several perspectives, including prescribed subgraphs such as paths, stars, caterpillars, spiders, and trees, special classes of host graphs such as bipartite and triangle-free graphs, and general degree conditions (see, for example, \cite{fk,ho,lyh,cfp,hl} and the survey \cite{tmsurvey}). 
Extending the edge-deletion perspective from a single edge to a subtree, Hasunuma \cite{hasunuma} posed the following conjecture.

\begin{conjecture}[Hasunuma~\cite{hasunuma}]\label{conj:hasunuma}
For integers $k,n\ge 1$, let $T$ be any tree of order $n$. If $G$ is a
$k$-connected graph with $\delta(G)\ge k+n-1$, then $G$ contains a copy $T'$ of
$T$ such that $G-E(T')$ remains $k$-connected.
\end{conjecture} 

Hasunuma proved the conjecture for $k\le2$ and also established it when the prescribed tree is a path. The case $k=3$ was proved independently by Liu, Liu, and Hong~\cite{llh} and by Yang and Tian~\cite{yt}.

\begin{theorem}[Hasunuma~\cite{hasunuma}]\label{thm:pathbound}
For integers $k,n\ge 1$, every $k$-connected graph $G$ with $\delta(G)\ge k+n-1$
contains a path $P$ of order $n$ such that $G-E(P)$ is $k$-connected.
\end{theorem}

 A matching $M$ in a $k$-connected graph $G$ is \emph{$k$-removable} if $G-M$ is $k$-connected. 
For an integer $\ell\ge1$, an \emph{$\ell$-matching} is a matching of size $\ell$.
In this paper, we study minimum degree conditions guaranteeing the existence of a $k$-removable $m$-matching for positive integers $k$ and $m$. 
This problem extends Theorem~\ref{thm:halin} from the deletion of a single edge to the deletion of $m$ pairwise disjoint edges.
A general sufficient condition follows from Theorem~\ref{thm:pathbound}. 
Since a path on $2m$ vertices contains an $m$-matching, Theorem~\ref{thm:pathbound} implies that for all $k,m\ge 1$, every $k$-connected graph $G$ with $\delta(G)\ge k+2m-1$ contains a $k$-removable $m$-matching. 
For $m=2$, Li, Zhou, Fujita, and Mao \cite{lzfm} recently obtained a  stronger result.

\begin{theorem}[Li, Zhou, Fujita, and Mao~\cite{lzfm}]\label{thm:lzfm}
For every $k\ge 1$, every $k$-connected graph $G$ with $\delta(G)\ge k+1$
contains a $k$-removable $2$-matching, unless $k=1$ and $G$ is a cycle.
\end{theorem}

They also established a result guaranteeing a removable matching whose size grows with the minimum degree.

\begin{theorem}[Li, Zhou, Fujita, and Mao~\cite{lzfm}]\label{thm:lzfm_delta}
For every $k\ge 4$, every $k$-connected graph $G$ with $\delta(G)\ge 3k-1$ contains a $k$-removable $\bigl\lceil(\delta(G)+1)/2\bigr\rceil$-matching, unless $G$ is the complete graph of odd order.
\end{theorem}

Motivated by these results, Li, Zhou, Fujita, and Mao introduced the parameter $f(k,\delta)$, defined as the largest integer $m$ such that every $k$-connected graph $G$ with $|V(G)|\ge 2\delta$ and $\delta(G)\ge \delta$ has a $k$-removable $m$-matching. 
Theorem~\ref{thm:lzfm_delta} states that $\bigl\lceil(\delta+1)/2\bigr\rceil\le f(k,\delta)$ for $\delta\ge 3k-1$ and $k\ge 4$. They further conjectured that the same lower bound holds under the optimal minimum degree assumption, except when $\delta(G)$ is even and $G \cong K_{\delta(G)+1}$.

\begin{conjecture}[Li, Zhou, Fujita, and Mao~\cite{lzfm}]\label{conj:lzfm}
For $k\ge 1$, every $k$-connected graph $G$ with $\delta(G)\ge k+1$ contains a
$k$-removable $\bigl\lceil(\delta(G)+1)/2\bigr\rceil$-matching, unless
$\delta(G)$ is even and $G\cong K_{\delta(G)+1}$.
\end{conjecture}
We note that Conjecture~\ref{conj:lzfm} requires a minor correction since every cycle is a counterexample to the conjecture, as noted in Theorem~\ref{thm:lzfm}.

Since every edge of a $k$-removable matching decreases the degree of each of its ends by one, the condition $\delta(G)\ge k+1$ is necessary. Although Conjecture~\ref{conj:lzfm} remains open, our main result substantially improves the condition $\delta(G)\ge k+2m-1$ obtained from Theorem~\ref{thm:pathbound}.

\begin{theorem}\label{thm:main:1}
For integers $k, m\ge 1$,  every $k$-connected graph $G$  with at least $2m$ vertices contains a $k$-removable $m$-matching if \[
   \delta(G)\ \ge\
   \begin{cases}
      \max\bigl\{k+\bigl\lceil\tfrac m2\bigr\rceil,\ 2m\bigr\}
         & \text{if } k\ge m,\\[2pt]
      k+m & \text{if } k<m.
   \end{cases}\]
\end{theorem}

Our result recovers Theorem~\ref{thm:halin} for $m=1$. 
The bound in Theorem~\ref{thm:main:1} strengthens all previously known bounds in \cite{lzfm} for $m\ge 3$ and $k\ge 4$ (see Figure~\ref{fig:comparison}). It also improves Theorem~\ref{thm:lzfm_delta} by replacing the assumptions $k\ge4$ and $\delta(G)\ge3k-1$ with the single condition $\delta(G)\ge2k+1$.
Consequently, Conjecture~\ref{conj:lzfm} holds whenever $\delta(G)\ge 2k+1$.
 
\begin{figure}[h!]
\centering
\begin{tikzpicture}[scale=1.05, x=0.52cm,y=0.48cm,>=stealth]

\draw[->,thick] (0,6) -- (13.3,6) node[right] {$m$};
\draw[->,thick] (0,6) -- (0,21.0) node[above] {$\delta$};
\draw[thick] (-0.15,6.35) -- (0.15,6.62);   
\draw[thick] (-0.15,6.55) -- (0.15,6.82);

\draw[dashed,gray!55] (1.6,6) -- (1.6,7);
\draw[dashed,gray!55] (4,6) -- (4,8);
\draw[dashed,gray!55] (6,6) -- (6,12);
\draw[dashed,gray!55] (9,6) -- (9,17);
\draw[dashed,gray!55] (0,7) -- (1,7);
\draw[dashed,gray!55] (0,8) -- (4,8);
\draw[dashed,gray!55] (0,12) -- (6,12);
\draw[dashed,gray!55] (0,17) -- (1,17);

\draw[very thick,black,dashed] (1,7) -- (7.75,20.5);
\node[black,font=\footnotesize,rotate=61,anchor=center] at (5.4,15.0)
      {Theorem~\ref{thm:pathbound}};

\draw[very thick,dashed,black] (1,17) -- (9,17) -- (10.75,20.5);
\node[black,font=\footnotesize,anchor=south west] at (1.15,17.05)
      {Theorem~\ref{thm:lzfm_delta}};
\node[black,font=\scriptsize,anchor=north] at (2.75,16.9) {($k\ge 4$)};
\node[black,font=\scriptsize,rotate=61,anchor=south] at (9.75,18.45) {slope $2$};

\draw[<->,gray!65] (4,8.3) -- (4,12.7);
\node[gray!75,font=\scriptsize,anchor=east] at (4.2,9.5) {$k-1$};
\draw[<->,gray!65] (9,15.25) -- (9,16.75);
\node[gray!75,font=\scriptsize,anchor=east] at (9.2,16.0) {$\tfrac{k}{2}-1$};

\draw[very thick,blue!65!black] (1,7) -- (1.6,7) -- (4,8) -- (6,12);   
\draw[very thick,blue!65!black] (6,12) -- (12.4,18.4);                 
\fill[blue!65!black] (1,7) circle (2pt) (1.6,7) circle (1.8pt) (4,8) circle (1.6pt)
                     (6,12) circle (2pt);
\node[blue!65!black,font=\footnotesize,anchor=north west] at (6.15,11.75)
      {Theorem~\ref{thm:main:1}};
\node[blue!65!black,font=\scriptsize,rotate=43,anchor=north] at (11.55,17.35) {slope $1$};

\node[below,font=\footnotesize] at (1,5.85)   {$1$};
\node[below,font=\footnotesize] at (1.6,5.85) {$2$};
\node[below,font=\footnotesize] at (4,5.85)   {$\tfrac{2k}{3}$};
\node[below,font=\footnotesize] at (6,5.85)   {$k$};
\node[below,font=\footnotesize] at (9,5.85)   {$\tfrac{3k}{2}$};

\node[left,font=\footnotesize] at (0,7)    {$k+1$};
\node[left,font=\footnotesize] at (0,8.1)  {$\tfrac{4k}{3}$};
\node[left,font=\footnotesize] at (0,12)   {$2k$};
\node[left,font=\footnotesize] at (0,17)   {$3k-1$};
\end{tikzpicture}
\caption{A schematic illustration of our main result, Theorem~\ref{thm:main:1} (blue line), alongside the earlier known bounds (thick dashed lines) on minimum degree forcing a $k$-removable $m$-matching, for fixed $k$, where floors and ceilings are suppressed.}
\label{fig:comparison}
\end{figure}
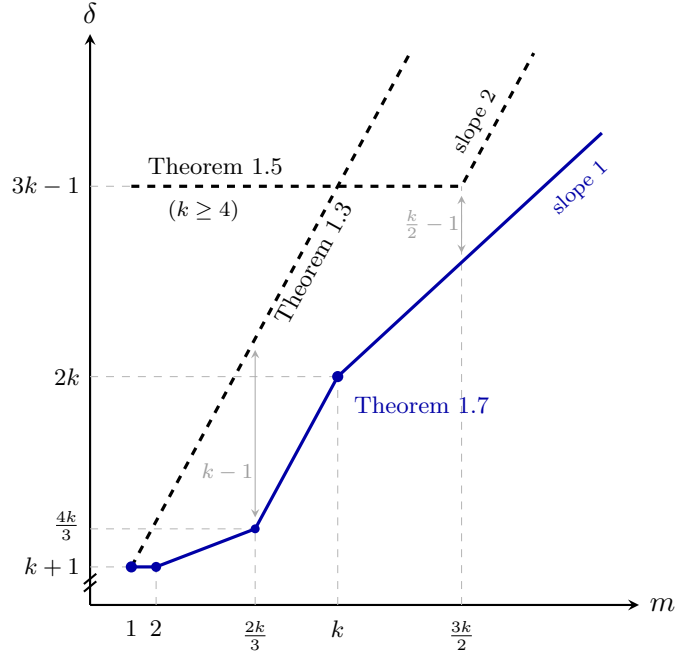

\begin{corollary}\label{cor1}
For every $k\ge1$, every $k$-connected graph $G$ with $\delta(G)\ge 2k+1$ contains a
$k$-removable $\bigl\lceil(\delta(G)+1)/2\bigr\rceil$-matching, unless $\delta(G)$ is even and $G\cong K_{\delta(G)+1}$.
\end{corollary}
\begin{proof}
Let $m=\bigl\lceil(\delta(G)+1)/2\bigr\rceil$. Since $\delta(G)\ge 2k+1$, we have $m>k$ and $\delta(G)\ge k+m$. Moreover, $|V(G)|\ge \delta(G)+2\ge 2m$, unless $\delta(G)$ is even and $G\cong K_{\delta(G)+1}$. Thus, Theorem~\ref{thm:main:1} yields a $k$-removable $m$-matching.
\end{proof}

The paper is organized as follows. In Section~\ref{sec:witness}, we introduce the notion of a \emph{witness} and analyze how two witnesses can overlap.
In Section~\ref{sec:key:lemma}, we prove a version of Theorem~\ref{thm:halin}  that guarantees a $k$-removable edge avoiding a prescribed vertex set. Finally, in Section~\ref{sec:proof}, we apply these results to prove Theorem~\ref{thm:main:1}. Section~\ref{sec:concluding} gives some remarks.

\section{Witnesses in $k$-connected graphs}\label{sec:witness}

Let $k$ be a positive integer, let $H$ be a $k$-connected graph, and let $xy$ be a non-$k$-removable edge of $H$.
Here, an edge is called \emph{non-$k$-removable} if its deletion does not preserve $k$-connectedness.
Hence there is a vertex cut $S$ of $H-xy$ of size at most $k-1$. 
Since $H$ is $k$-connected, $S$ separates $x$ and $y$ in $H-xy$.
If $X$ and $Y$ are the components of $(H-xy)-S$ containing $x$ and $y$, respectively, 
then we call the ordered $5$-tuple $(x,y,S,X,Y)$ a \emph{witness} for $xy$.

\begin{observation}\label{obs:witness}
    Let $H$ be a $k$-connected graph. 
    For a non-$k$-removable edge $xy$ of $H$, let $(x,y,S,X,Y)$ be a witness for $xy$. 
    Then $|S|=k-1$, $xy$ is a cut edge of $H-S$, and $(H-xy)-S$ has exactly two components $X$ and $Y$.
\end{observation}

\begin{proof}
Since $H-S$ is connected whereas $(H-xy)-S=(H-S)-xy$ is not connected, $xy$ is a cut edge of $H-S$ joining $X$ and $Y$.
So, $(H-xy)-S$ has exactly two components $X$ and $Y$. 
 If $X=\{x\}$ and $Y=\{y\}$, then $|S|=|V(H)|-2\ge k-1$ since a $k$-connected graph has at least $k+1$ vertices. 
Otherwise, $|X|\ge2$ or $|Y|\ge2$, and then $S\cup\{x\}$ or $S\cup\{y\}$, respectively, is a vertex cut of $H$.
Hence $|S|\ge k-1$ in either case.
\end{proof}

For a witness $(x,y,S,X,Y)$ for $xy$, the edge $xy$ is the only edge of $H-S$ joining $X$ and $Y$.
Moreover, $N_H(x)\subseteq X \cup S \cup \{y\}$ and $N_H(z)\subseteq X \cup S$ for every $z \in X$ distinct from $x$, where these symmetrically hold for $y$ and $Y$.
The following lemma describes a structure when two witnesses overlap in a $k$-connected graph.

\begin{lemma}\label{lem:uncrossing}
Let $H$ be a $k$-connected graph for a positive integer $k$.
Suppose that $(u,v,T,A,B)$ and $(u',u,T',A',B')$ are witnesses for two distinct non-$k$-removable edges $uv,\,u'u\in E(H)$. 
Define
\[
  T_1=T\cap A',\quad T_2=T\cap B',\quad T'_1=T'\cap A,\quad T'_2=T'\cap B,
\]
and, for $i,j\in\{1,2\}$, set $Q_{ij}=(T\cap T')\cup T_i\cup T'_j$ (see
Figure~\ref{fig:crossing-witnesses}). 
Then the following hold.
\begin{itemize}
\item[(i)] If $|Q_{21}|\le k-2$, then $d_H(u)\le k$.
\item[(ii)] If $|Q_{21}|\ge k-1$, then $B\cap A'=\emptyset$. In addition, if $|Q_{11}|\ge k$, then $B\cap B'=\emptyset$, $|A|>|B|$, and $|A|>|A'|$ or $|A|>|B'|$.
\end{itemize}
\end{lemma}

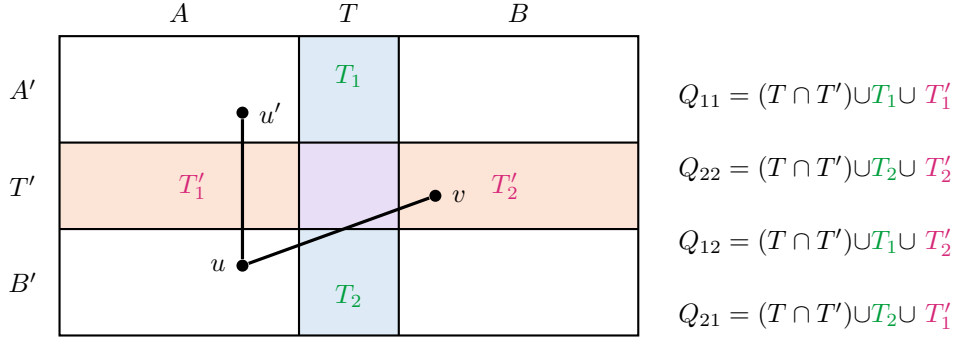
\begin{figure}[ht]
\centering
\begin{tikzpicture}[scale=0.88, every node/.style={font=\small}]
\definecolor{cutblue}{RGB}{222,235,247}
\definecolor{cutred}{RGB}{252,228,214}
\definecolor{overlappurple}{RGB}{232,222,245}
\definecolor{qmagenta}{RGB}{214,39,120}
\definecolor{qgreen}{RGB}{0,158,60}
\fill[cutblue] (3.6,1.0) rectangle (5.1,5.5);
\fill[cutred]  (0,2.6) rectangle (8.7,3.9);
\fill[overlappurple] (3.6,2.6) rectangle (5.1,3.9);
\draw[thick] (0,1.0) rectangle (8.7,5.5);
\draw[thick] (3.6,1.0) -- (3.6,5.5);
\draw[thick] (5.1,1.0) -- (5.1,5.5);
\draw[thick] (0,2.6) -- (8.7,2.6);
\draw[thick] (0,3.9) -- (8.7,3.9);
\node at (1.8,5.85) {$A$};
\node at (4.35,5.85) {$T$};
\node at (6.9,5.85) {$B$};
\node at (-0.55,4.7)  {$A'$};
\node at (-0.55,3.25) {$T'$};
\node at (-0.55,1.8)  {$B'$};
\node[circle,fill=black,inner sep=1.6pt,label=right:$u'$] (up) at (2.75,4.35) {};
\node[circle,fill=black,inner sep=1.6pt,label=left:$u$]  (u)  at (2.75,2.05) {};
\node[circle,fill=black,inner sep=1.6pt,label=right:$v$] (v)  at (5.65,3.1) {};
\draw[very thick] (u) -- (v);
\draw[very thick] (up) -- (u);
\node[qmagenta] at (2.0,3.25) {$T'_1$};
\node[qmagenta] at (6.7,3.25) {$T'_2$};
\node[qgreen]   at (4.35,4.9) {$T_1$};
\node[qgreen]   at (4.35,1.6) {$T_2$};
\node[align=left,anchor=west] at (9.15,4.6)
    {$Q_{11}=(T\cap T')\cup$\textcolor{qgreen}{$T_1$}$\cup$ \textcolor{qmagenta}{$T'_1$}};
\node[align=left,anchor=west] at (9.15,3.5)
  {$Q_{22}=(T\cap T')\cup$\textcolor{qgreen}{$T_2$}$\cup$ \textcolor{qmagenta}{$T'_2$}};
\node[align=left,anchor=west] at (9.15,2.4)
 {$Q_{12}=(T\cap T')\cup$\textcolor{qgreen}{$T_1$}$\cup$ \textcolor{qmagenta}{$T'_2$}};
\node[align=left,anchor=west] at (9.15,1.3)
  {$Q_{21}=(T\cap T')\cup$\textcolor{qgreen}{$T_2$}$\cup$ \textcolor{qmagenta}{$T'_1$}};
\end{tikzpicture}
\caption{The two witnesses $(u,v,T,A,B)$ and $(u',u,T',A',B')$. The vertical
strip is the vertex cut $T$ in $H-uv$, and the horizontal strip is the vertex cut
$T'$ in $H-u'u$. The vertex $u'$ lies in $A'$ and outside $B$; it may lie in $A$
or in $T$.}
\label{fig:crossing-witnesses}
\end{figure}

\begin{proof}
By the definition of a witness, $uv$ (resp.\ $u'u$) is the only edge of $H-T$
(resp.\ $H-T'$) joining $A$ and $B$ (resp.\ $A'$ and $B'$), and therefore $u' \not\in B$ (resp.\  $v \not\in A'$).

We first consider the case where $|Q_{21}|\le k-2$.
Then $|Q_{21}\cup\{u\}| \le k-1$.
If $(A\cap B')\setminus\{u\}\ne\emptyset$, then $Q_{21}\cup\{u\}$ separates
$(A\cap B')\setminus\{u\}$ from $v$, which contradicts the $k$-connectedness of $H$.
Thus $A\cap B'=\{u\}$, so $N_H(u) \subseteq Q_{21}\cup\{u',v\}$.
Therefore $d_H(u)\le k$ and (i) holds.

Now, we consider the case where $|Q_{21}|\ge k-1$.
Note that
\[ |Q_{11}|+|Q_{22}|=|Q_{12}|+|Q_{21}|=|T|+|T'| = 2k-2.\]
Thus $|Q_{12}|\le k-1$.
If $B\cap A'\ne\emptyset$, then, since $v \not\in A'$, $Q_{12}$ is a vertex cut of $H$, which is impossible.
Therefore $B\cap A'=\emptyset$.
Assume that $|Q_{11}|\ge k$.
Then $|Q_{22}|\le k-2$, so $|Q_{22}\cup\{u\}|\le k-1$.
If $B\cap B'\ne\emptyset$, then $Q_{22}\cup\{u\}$ separates $B\cap B'$ from $u'$, which is impossible.
Thus $B\cap B'=\emptyset$, so $v\in T'$. 
Moreover, $B=T_2'$. 
Since $|Q_{11}|\ge k>|Q_{12}|$, we have $|T'_1|>|T'_2|$, so $|A|\ge |T'_1|>|T'_2| = |B|$.
Then $2|A|>|A|+|B|=|V(H)|-|T|=|V(H)|-|T'|=|A'|+|B'|$ by Observation~\ref{obs:witness},  so $|A|>|A'|$ or $|A|>|B'|$.
Thus (ii) holds. 
\end{proof}

\section{$k$-removable edge avoiding a prescribed set}\label{sec:key:lemma}

The following theorem produces, in a $k$-connected graph, a $k$-removable edge avoiding a prescribed vertex set $W$ provided that $H-W$ has an edge. It plays a central role in the proof of Theorem~\ref{thm:main:1}.

\begin{theorem}\label{lem:key}
Let $H$ be a $k$-connected graph for a positive integer $k$, and let $W$ be a possibly empty subset of $V(H)$. Set $w=|W|$.
Suppose that $E(H-W)\neq\emptyset$.
Then $H$ has a $k$-removable edge with both ends in $V(H)\setminus W$ if $d_H(x)\ge h(k,w)$ for every $x\in V(H)\setminus W$, where
\[
    h(k,w):= \max \Big\{ k+\Bigl\lceil\tfrac{w+1}{4}\Bigr\rceil,           \min \Big\{ k+\Bigl\lceil\tfrac{w+1}{2}\Bigr\rceil,w+2            \Big\} \Big\} .
\]
\end{theorem}

Before proving Theorem~\ref{lem:key}, we observe that $h(k,w)$ is non-decreasing in  $w$ (see Figure~\ref{fig:graph:h}), since 
\[
h(k,w)=
\begin{cases}
 k+\Bigl\lceil\tfrac{w+1}{4}\Bigr\rceil
  & \text{if } w \le  \Big\lfloor\tfrac   {4(k-2)}3 \Big\rfloor ,\\[1.2ex]
w+2
 & \text{if } \Big\lfloor\tfrac   {4(k-2)}3 \Big\rfloor    <  w \le 2k-4,\\[1.2ex]
k+\Bigl\lceil\tfrac{w+1}{2}\Bigr\rceil &\text{if }w \ge 2k-3.
\end{cases}
\]

 \begin{figure}[ht]
\centering
\begin{tikzpicture}[x=0.62cm,y=0.62cm,>=stealth]
\def\xa{4}   
\def\xb{6}   
\draw[->,thick] (-0.4,0) -- (14.4,0) node[right] {$w$};
\draw[->,thick] (0,-0.4) -- (0,9.0) node[above] {};
\draw[dashed,gray!60] (\xa,0) -- (\xa,2.5);
\draw[dashed,gray!60] (\xb,0) -- (\xb,4.5);
\draw[very thick,blue!65!black]
      (0,1.5) -- (4,2.5) -- (6,4.5) -- (13,8.0);
\fill[blue!65!black] (0,1.5) circle (1.7pt) (4,2.5) circle (1.7pt)
                     (6,4.5) circle (1.7pt);
\node[above=1pt,font=\small,blue!65!black] at (2,2.0)    {slope $\tfrac14$};
\node[above left=1pt and -2pt,font=\small,blue!65!black] at (5,3.5) {slope $1$};
\node[above=1pt,font=\small,blue!65!black] at (9.5,6.5) {slope $\tfrac12$};
\node[below right=2pt and -12pt,font=\footnotesize] at (2,2.0)
      {$h(k,w)=k+\bigl\lceil\tfrac{w+1}{4}\bigr\rceil$};
\node[below right=2pt and 0pt,font=\footnotesize] at (4.4,3.4) {$h(k,w)=w+2$};
\node[below right=2pt and -14pt,font=\footnotesize] at (9.5,6.25)
      {$h(k,w)=k+\bigl\lceil\tfrac{w+1}{2}\bigr\rceil$};
\node[below,font=\footnotesize] at (\xa-0.1,-0.15) {$\bigl\lfloor\tfrac{4(k-2)}{3}\bigr\rfloor$};
\node[below,font=\footnotesize] at (\xb+0.3,-0.15) {$2k-4$};
\node[left,font=\small] at (0,1.5) {$k+1$};
\end{tikzpicture}
\caption{A schematic illustration of $h(k,w)$ for fixed $k$, where floors and ceilings are suppressed.}
\label{fig:graph:h}
\end{figure}
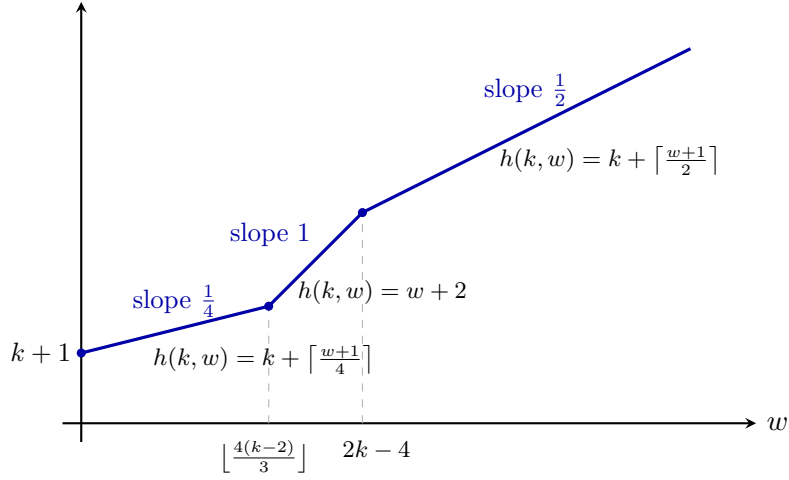

\begin{proof}
Let $U=V(H)\setminus W$. Then $E(H[U])\neq\emptyset$ by assumption.
Suppose to the contrary that no edge of $H[U]$ is
$k$-removable in $H$. 
For every $xy\in E(H[U])$ and its witness $(x,y,S,X,Y)$, we have
$|W\cap X|+|W\cap Y|\le|W|\le w$, so $X$ or $Y$ meets $W$ in at most
$\lfloor w/2\rfloor$ vertices.  
A witness $(x,y,S,X,Y)$ for $xy$ is called {\em good} if \[   |W\cap X|\ \le\ \Bigl\lfloor\tfrac w2\Bigr\rfloor.\]
Among all good witnesses for edges in $H[U]$, we choose $(u,v,T,A,B)$ with $|A|$ minimum. Note that $u,v\in U$ and $|T|=k-1$. 
\begin{claim}\label{clm:no_neighbor_ATU}
$u$ has no neighbor in $(A\cup T)\cap U$.
\end{claim}
\begin{proof} Suppose that $u$ has a neighbor in $(A\cup T) \cap U$.
If $u$ has a neighbor in $A\cap U$, then choose a neighbor $u'$ of $u$ in $A\cap U$.
Otherwise, let $u'$ be a neighbor of $u$ in $T\cap U$.
As $u'u\in E(H[U])$ is not $k$-removable, there is a witness $(u',u,T',A',B')$ for $u'u$. 
We adopt the notation in Lemma~\ref{lem:uncrossing}.
We first show that $B\cap A'=A\cap A'=\emptyset$, and so $u'\in T_1$, as shown in  Figure~\ref{fig:crossing-witnesses2}.

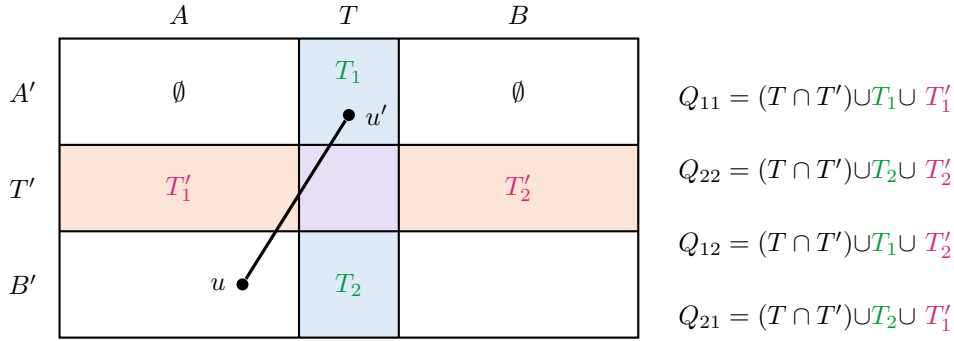
\begin{figure}[ht]
\centering
\begin{tikzpicture}[scale=0.88, every node/.style={font=\small}]
\definecolor{cutblue}{RGB}{222,235,247}
\definecolor{cutred}{RGB}{252,228,214}
\definecolor{overlappurple}{RGB}{232,222,245}
\definecolor{qmagenta}{RGB}{214,39,120}
\definecolor{qgreen}{RGB}{0,158,60}
\fill[cutblue] (3.6,1.0) rectangle (5.1,5.5);
\fill[cutred]  (0,2.6) rectangle (8.7,3.9);
\fill[overlappurple] (3.6,2.6) rectangle (5.1,3.9);
\draw[thick] (0,1.0) rectangle (8.7,5.5);
\draw[thick] (3.6,1.0) -- (3.6,5.5);
\draw[thick] (5.1,1.0) -- (5.1,5.5);
\draw[thick] (0,2.6) -- (8.7,2.6);
\draw[thick] (0,3.9) -- (8.7,3.9);
\node at (1.8,5.85) {$A$};
\node at (4.35,5.85) {$T$};
\node at (6.9,5.85) {$B$};
\node at (-0.55,4.7)  {$A'$};
\node at (-0.55,3.25) {$T'$};
\node at (-0.55,1.8)  {$B'$};
\node at (1.8,4.7) {$\emptyset$};
\node at (6.9,4.7) {$\emptyset$};
\node[circle,fill=black,inner sep=1.6pt,label=right:$u'$] (up) at (4.35,4.35) {};
\node[circle,fill=black,inner sep=1.6pt,label=left:$u$]  (u)  at (2.75,1.8) {};
\draw[very thick] (up) -- (u);
\node[qmagenta] at (1.8,3.25) {$T'_1$};
\node[qmagenta] at (6.9,3.25) {$T'_2$};
\node[qgreen]   at (4.35,5) {$T_1$};
\node[qgreen]   at (4.35,1.8)   {$T_2$};
\node[align=left,anchor=west] at (9.15,4.6)
    {$Q_{11}=(T\cap T')\cup$\textcolor{qgreen}{$T_1$}$\cup$ \textcolor{qmagenta}{$T'_1$}};
\node[align=left,anchor=west] at (9.15,3.5)
  {$Q_{22}=(T\cap T')\cup$\textcolor{qgreen}{$T_2$}$\cup$ \textcolor{qmagenta}{$T'_2$}};
\node[align=left,anchor=west] at (9.15,2.4)
 {$Q_{12}=(T\cap T')\cup$\textcolor{qgreen}{$T_1$}$\cup$ \textcolor{qmagenta}{$T'_2$}};
\node[align=left,anchor=west] at (9.15,1.3)
  {$Q_{21}=(T\cap T')\cup$\textcolor{qgreen}{$T_2$}$\cup$ \textcolor{qmagenta}{$T'_1$}};
\end{tikzpicture}
\caption{The two witnesses $(u,v,T,A,B)$ and $(u',u,T',A',B')$ in the proof of Claim~\ref{clm:no_neighbor_ATU}.}
\label{fig:crossing-witnesses2}
\end{figure}

If $|Q_{21}|\le k-2$, then $A\cap B'=\{u\}$ since  otherwise $Q_{21}\cup\{u\}$ would be a vertex cut of size at most $k-1$ 
separating $v$ from $(A\cap B') -\{u\}$. However, this also yields a contradiction since $d_H(u)\le |Q_{21}\cup\{u',v\}|\le k$.
Thus $|Q_{21}|\ge k-1$.
By Lemma~\ref{lem:uncrossing}(ii), $B\cap A'=\emptyset$.

Suppose to the contrary that $A\cap A'\neq \emptyset$. Suppose $|Q_{11}|\ge k$.
By Lemma~\ref{lem:uncrossing}(ii) again, $B\cap B'=\emptyset$ and $|A|>|B|$.
Then the witness $(v,u,T,B,A)$ satisfies $|B|<|A|$, so by the minimality of $|A|$ it is not a good witness, that is, $|W\cap B|\ge\lfloor w/2\rfloor+1$. Hence
\[
   |W\cap(A\cup T)|=|W|-|W\cap B|
   \le w-\Bigl(\Bigl\lfloor\tfrac w2\Bigr\rfloor+1\Bigr)
   \le \Bigl\lfloor\tfrac w2\Bigr\rfloor .
\]
Since $B\cap A'=B\cap B'=\emptyset$, we have $|W\cap A'|\le\lfloor w/2\rfloor$ and $|W\cap B'|\le\lfloor w/2\rfloor$. 
Thus both $(u',u,T',A',B')$ and $(u,u',T',B',A')$ are good.
Then the minimality of $|A|$ yields $|A|\le|A'|$ and $|A|\le|B'|$, contradicting Lemma~\ref{lem:uncrossing}(ii).
Hence $|Q_{11}|\le k-1$.
If $u'\in A$, $Q_{11}$ gives a good witness for $u'u$, which contradicts the minimality of $|A|$. 
If $u' \not\in A$, then $Q_{11}$ is a vertex cut of $H$, which contradicts the $k$-connectedness of $H$.
Therefore $A\cap A'=\emptyset$, which implies that $u' \in T$.

Since $A'=T_1$ and $d_H(u')\ge h(k,w)\ge k+\lceil (w+1)/4\rceil$, 
it follows that $N_H(u')\subseteq (T_1\setminus\{u'\})\cup T'\cup\{u\}$, and we obtain 
\[
   |T_1|\ge |N_H(u')\cap T_1|+1\ \ge\ d_H(u')-|T'|\ \ge\ \Bigl\lceil \tfrac{w+1}{4}\Bigr\rceil+1.
\]
Since $u' \in T$, it follows that $u$ has no neighbor in $A\cap U$ by the choice of $u'$.
Note that $uu'$ is the only edge joining $A'$ and $B'$ in $H-T'$.
Thus no vertex in $T_1-\{u'\}$ is adjacent to $u$, so 
\[
   N_H(u)\ \subseteq\ (A\cap W) \cup \bigl(T\setminus T_1 \bigr) \cup \{u', v\} .
\]
Hence, since $(u,v,T,A,B)$ is good and $|T|=k-1$,
\[
   d_H(u)\ \le\ \Bigl\lfloor\tfrac w2\Bigr\rfloor
              +\Bigl((k-1)-\Bigl\lceil \tfrac{w+1}{4}\Bigr\rceil-1 \Bigr)+2
          \ =\ k+\Bigl\lfloor\tfrac w2\Bigr\rfloor-\Bigl\lceil \tfrac{w+1}{4}\Bigr\rceil < k+\Bigl\lceil \tfrac{w+1}{4}\Bigr\rceil \le h(k,w),
\]
a contradiction. 
\end{proof}

Since $uv$ is a cut edge of $H-T$, the claim above gives
\[   N_H(u)\ \subseteq\ (A\cap W) \cup (T\cap W) \cup \{v\}.\]
We bound $d_H(u)$ in two ways.
First, since $A$, $T$, and $\{v\}$ are pairwise disjoint, 
\[
   d_H(u)\ \le\ |A\cap W|+|T|+1\ \le\ \Bigl\lfloor\tfrac w2\Bigr\rfloor+(k-1)+1
         \ =\ k+\Bigl\lfloor\tfrac w2\Bigr\rfloor .
\]
Second, since $(A\cap W)\cup(T\cap W)\subseteq W$, $d_H(u)  \le  |W|+1  \le  w+1$.
Hence
$d_H(u)\le\min\bigl\{\,k+\lfloor w/2\rfloor,\ w+1\,\bigr\} < h(k,w)$, a contradiction.  
\end{proof}

\section{Proof of Theorem~\ref{thm:main:1}}\label{sec:proof} 
In this section, we prove Theorem~\ref{thm:main:1}.
We need the following simple observation. 

\begin{observation}\label{lem:k_con_preserv}
Let $G$ be a $k$-connected graph for a positive integer $k$.
If $G'$ is the graph obtained from $G$ by adding a new vertex joined to at least $k$ vertices of $G$, then $G'$ is $k$-connected.
\end{observation}

\begin{proof}[Proof of Theorem~\ref{thm:main:1}]
Let $G$ be a $k$-connected graph with at least $2m$ vertices and \[
   \delta(G)\ \ge\
   \begin{cases}
      \max\bigl\{k+\bigl\lceil\tfrac m2\bigr\rceil,\ 2m\bigr\}
         & \text{if } k\ge m,\\[2pt]
      k+m & \text{if } k<m.
   \end{cases}
\]
Among all $k$-removable matchings of $G$, we choose one of maximum size, say $M$, and set $s=|M|$ and $W=V(M)$. Within this proof, we regard the empty set as a $k$-removable matching. Hence such a maximum matching $M$ is well-defined.
In particular, if $G$ has no $k$-removable edges, then $M=W=\emptyset$ and $s=0$. 
It suffices to show $s\ge m$, so suppose to the contrary that $s\le m-1$. 
Let $H=G-M$ and $U=V(G)\setminus W$. 
Then $H$ is $k$-connected and $|W|=2s$.

\begin{claim}\label{clm:good}
$M$ can be chosen so that $E(H-W)\neq\emptyset$.
\end{claim}

\begin{proof}
Note that $H-W=H[U]=G[U]$.
Suppose to the contrary that $G[U]$ is edgeless.
Since $|V(G)| \ge 2m$ and $|W|\le 2m-2$, $U$ has two distinct vertices $u$ and $v$.
As $U$ is independent, $N_G(u),N_G(v)\subseteq W$. 
In particular, $\delta(G) \le d_G(u) \le 2m-2$.
By the minimum degree condition of $G$, $m>k$ and $\delta(G)\ge k+m$. 

If $|N_G(u)\cap\{a,b\}|+|N_G(v)\cap \{a,b\}| \le 2$ for each edge $ab \in M$, then, since $N_G(u),N_G(v)\subseteq W$,
\[2(k+m)\le 2\delta(G)\le d_G(u)+d_G(v)\le 2|M| \le 2m-2,\]
a contradiction.
Thus there exists an edge $xy \in M$ such that $|N_G(u)\cap\{x,y\}|+|N_G(v)\cap \{x,y\}| \ge 3$.
Hence, by symmetry, we may assume $ux, yv \in E(G)$.

Note that $H$ is $k$-connected, $|V(H)|\ge k+m+1>k+1$, and 
\[
d_{H}(z)\ge
\begin{cases}
k+m-1,& \text{if } z\in W,\\
k+m,& \text{if } z\in U.
\end{cases}
\]
We will show that  $H-u$ is $k$-connected.
Suppose to the contrary that $H-u$ is not $k$-connected. Since $H-u$ has at least $k+1$ vertices, there is a vertex cut $S$ of $H$ such that $|S|=k$ and $u\in S$.
Choose a component $A$ of $H-S$ such that $|A \cap W|$ is minimum.
Since $H-S$ has at least two components and $|W| \le 2m-2$, we have $|A \cap W| \le m-1$.
If $A\cap U$ has a vertex $z$, then, since $U$ is independent, $N_H(z) \subseteq S \cup (A \cap W)$, so $d_H(z)\le k+m-1$, a contradiction.
Thus $A\cap U=\emptyset$ and so $A \subseteq W$.
Consequently, for any $z\in A$, $d_H(z)\le |S|+|A|-1 \le k+m-2$, a contradiction.
Therefore $H-u$ is $k$-connected.

Since $H-u$ is $k$-connected and $d_{H-ux}(u) \ge d_H(u)-1 \ge k+m-1 \ge k$, Observation~\ref{lem:k_con_preserv} implies that $H-ux$ is $k$-connected.
Hence, the graph $H'$ obtained from $H-ux$ by adding the edge $xy$  is also $k$-connected.
Now, let $M'=(M-xy)\cup \{ux\}$. 
Since $u \not\in W$, $M'$ is an $s$-matching and $G-M'=H'$. 
Thus $M'$ is $k$-removable, and hence is also of maximum size among $k$-removable matchings of $G$.
Moreover, $V(M')=(W\setminus\{y\})\cup\{u\}$ and $yv \in \bigl(G-M'\bigr)-V(M')$. 
Replacing $M$ by $M'$ yields the desired choice of $M$.
\end{proof}

By~Claim~\ref{clm:good}, $E(H[U])\neq \emptyset$.
We apply Theorem~\ref{lem:key} to $H$ and $W$ with $w=2s$. For every $x\in U$, we
have $d_H(x)=d_G(x)\ge\delta(G)$, and since $h(k,w)$ is non-decreasing in $w$ and
$2s\le 2m-2$,
\[
   h(k,2s)\ \le\ h(k,2m-2)
   \ =\ \max\Bigl\{\,k+\Bigl\lceil\tfrac m2\Bigr\rceil,\
   \min\{k+m,\ 2m\}\,\Bigr\}\ \le\ \delta(G),
\]
where the last inequality holds since if $k\ge m$, both terms are at most $\max\{k+\lceil m/2\rceil,\,2m\}$, and otherwise, both terms are at most $k+m$.
Hence Theorem~\ref{lem:key} yields a $k$-removable edge $e$ of $H$ with both ends
in $U$. Then $M\cup\{ e \}$ is a $k$-removable $(s+1)$-matching in $G$,
which contradicts the maximality of $M$.
\end{proof}

\section{Concluding remarks}\label{sec:concluding}
We revisit our main theorem and compare it with known results. 
Recall that 
$f(k,\delta)$ is the largest integer $m$ such that every $k$-connected graph $G$ with $|V(G)|\ge 2\delta$ and $\delta(G)\ge \delta$ has a $k$-removable $m$-matching. 
Theorem~\ref{thm:main:1} bounds the function $f(k,\delta)$.
As observed in~\cite{lzfm},  $f(k,\delta) \le \delta$ and the upper bound is  attained by the graph $K_{\delta,\,n-\delta}$. 

\begin{corollary}\label{cor:f-bound}
For all integers $\delta > k \ge 1$,
\[
f(k,\delta) \;\ge\;
\begin{cases}
\ \delta-k & \text{if } \delta \ge 2k+1,\\[4pt]
\ \min\bigl\{\lfloor \delta/2 \rfloor,\; 2(\delta-k)\bigr\} & \text{if } k+1 \le \delta \le 2k.
\end{cases}
\]
\end{corollary}

\begin{proof}
Let $G$ be a $k$-connected graph with $|V(G)| \ge 2\delta$ and $\delta(G) \ge \delta$,
and let $m$ be the right-hand side. Then $m \le \delta$, so $|V(G)| \ge 2\delta \ge 2m$.
If $\delta \ge 2k+1$, then $m = \delta-k > k$ and $k+m = \delta$. If $\delta \le 2k$,
then $\max\{k+\lceil m/2 \rceil,\, 2m\} \le \delta$. 
Theorem~\ref{thm:main:1} yields a $k$-removable $m$-matching.
\end{proof}

Table~\ref{tab:f} summarizes the resulting bounds. 
For $k\ge 4$, Corollary~5.1 improves the bounds in~\cite{lzfm}
when $\delta\ge k+2$, and matches their bound when $\delta=k+1$.
\begin{table}[H]
\centering
\begin{tabular}{c|c||c|c|c}
\hline
$k$ & $\delta$ & lower bound & upper bound & source of lower bound \\
\hline
\hline
\multirow{2}{*}{$1$} & $2$ & $1$ & $1$ & \cite{halin} \\
  & $\ge 3$ & $\delta$ & $\delta$ & \cite{lzfm} \\
\hline
\multirow{3}{*}{$2$} & $3$ & $2$ & $2$ & \cite{lzfm} \\
 & $4$ & $3$  & $4$ & \cite{lzfm} \\
 & $\ge 5$ & $\delta-2$ & $\delta$ & Cor.~\ref{cor:f-bound}; \cite{lzfm} if $\delta=5$ or even \\
\hline
\multirow{3}{*}{$3$} & $4$ & $2$ & $3$ & Cor.~\ref{cor:f-bound};\cite{lzfm} \\
 & $5,6$ & $\lceil (\delta+1)/2 \rceil$ & $\delta$ & \cite{lzfm} \\
 & $\ge 7$ & $\delta-3$ & $\delta$ & Cor.~\ref{cor:f-bound};  \cite{lzfm} if $\delta \in \{7,8\}$ \\
\hline
\multirow{3}{*}{$\ge 4$} & $k+1$ & $2$ & $k$ & Cor.~\ref{cor:f-bound}; \cite{lzfm} \\
 & $k+2 \le \delta \le 2k$ & $\min\{\lfloor \delta/2 \rfloor,\, 2(\delta-k)\}$ & $\delta$ & Cor.~\ref{cor:f-bound} \\
 & $\ge 2k+1$ & $\delta-k$ & $\delta$ & Cor.~\ref{cor:f-bound} \\
\hline
\end{tabular}
\caption{Best known bounds on $f(k,\delta)$ for $\delta > k \ge 1$. Both upper bounds
$f(k,\delta) \le \delta$ and $f(k,k+1) \le k$ are from~\cite{lzfm}.}
\label{tab:f}
\end{table}

\section*{Acknowledgements}
Hojin Chu was supported by a KIAS Individual Grant (CG101801) at Korea Institute
for Advanced Study.
Ringi Kim was supported by the National Research Foundation of Korea (NRF) grant
funded by the Korea government (MSIT) (No.~RS-2025-00561867), and supported by
INHA UNIVERSITY Research Grant.
Boram Park was supported by the National Research Foundation of Korea (NRF) grant
funded by the Korea government (MSIT) (No.~RS-2025-00523206), and supported by the
New Faculty Startup Fund from Seoul National University.

\end{document}